\documentclass[submission]{FPSAC2025}

\usepackage{lipsum}
\usepackage{float}
\usepackage[utf8]{inputenc}

\usepackage{capt-of}

\usepackage{geometry}                
\usepackage{graphicx}
\usepackage{mathrsfs}
\usepackage{soul}
\usepackage{amssymb}
\usepackage{amsmath}
\usepackage{amsfonts}
\usepackage{mathrsfs}
\usepackage{epstopdf}
\usepackage{lscape}
\usepackage{array}
\usepackage[utf8]{inputenc}
\usepackage{tikz,caption}

\usepackage{enumitem}
\setlist[itemize]{leftmargin=2em}
\setlist[enumerate]{leftmargin=2em}

\usepackage{booktabs}
\usepackage{multirow}
\usepackage{mathtools}
\usepackage[linesnumbered,ruled]{algorithm2e}
\usepackage{tikz}
\usepackage{mathtools}
\usepackage{soul}
\usepackage{upgreek}

\newtheorem{theorem}{Theorem}[section]
\newtheorem{lemma}[theorem]{Lemma}
\newtheorem{corollary}[theorem]{Corollary}

\newtheorem{proposition}[theorem]{Proposition}

\theoremstyle{definition}
\newtheorem{definition}[theorem]{Definition}
\newtheorem{example}[theorem]{Example}

\newtheorem{remark}[theorem]{Remark}

\newcolumntype{C}[1]{>{\centering\arraybackslash}m{#1}}

\newcommand{\csff}{\mathbf{X}_F}
\newcommand{\csft}{\mathbf{X}_T}
\newcommand{\stfrak}{\mathfrak{st}}
\newcommand{\lmulti }{\{\!\!\{}
\newcommand{\rmulti}{\}\!\!\}}

\newcommand{\sort}{\mathrm{sort}}
\newcommand{\lc}{\lambda_\mathrm{LC}}
\newcommand{\lead}{\lambda_\mathrm{lead}}
\newcommand{\calit}{\mathcal{I}_T}

\definecolor{darkblue}{rgb}{0.0,0,0.7} 
\definecolor{darkred}{rgb}{0.7,0,0} 
\definecolor{darkgreen}{rgb}{0, .6, 0} 

\newcommand{\defncolor}{\color{darkred}}
\newcommand{\defn}[1]{{\defncolor\emph{#1}}} 

\usepackage[colorinlistoftodos]{todonotes}

\title[Tree Reconstruction]{On the reconstruction of trees from their chromatic symmetric functions}

\author[Gonzalez, Orellana, and Tomba]{Michael Gonzalez \addressmark{1}\thanks{\href{mailto:mikegonz1130@gmail.com}{mikegonz1130@gmail.com}}, Rosa Orellana \addressmark{1}\thanks{\href{mailto:Rosa.C.Orellana@Dartmouth.edu}{Rosa.C.Orellana@Dartmouth.edu}, partially supported by NSF grant DMS-2153998} \and Mario Tomba\addressmark{1}\thanks{\href{mailto:mario.tomba.morillo.25@dartmouh.edu}{mario.tomba.morillo.25@dartmouh.edu}} }

\address{\addressmark{1}Mathematics Department, Dartmouth College, Hanover, NH 03755, U.S.A.}

\received{\today}

\abstract{
We study Stanley's chromatic symmetric function (CSF) for trees when expressed in the star basis.  We use the deletion-near-contraction (DNC) algorithm recently introduced in \cite{ADOZ} to compute coefficients that occur in the CSF in the star basis. In particular, one of our main results determines the smallest partition in lexicographic order that occurs as an indexing partition in the CSF, and we also give a formula for its coefficient. In addition to describing properties of trees encoded in the coefficients of the star basis, we give an algorithm for reconstructing trees of diameter less than six. 
}

\keywords{symmetric chromatic polynomial, trees, graph coloring}

\usepackage[backend=bibtex]{biblatex}
\begin{document}

\maketitle

\section{Introduction}
In 1995, Stanley \cite{stanley95symmetric} introduced the \emph{chromatic symmetric function} (CSF) for a finite, simple graph $G$, $\mathbf{X}_G$. The CSF is a multivariate generalization of the chromatic polynomial of a graph that counts the number or proper colorings of the vertices of $G$. This function has garnered significant interest, \cite{aliste2014proper, crew2020deletion, Crew-distinguishing, DSvW-epos,heil2019algorithm, loebl2018isomorphism, martin2008distinguishing, GS-noncom, LW-rootedversion}. This interest is partly due to the   
 \emph{Tree Isomorphism Conjecture}, which states that the CSF distinguishes non-isomorphic trees. This conjecture is known to hold for trees with less than 30 vertices  \cite{heil2019algorithm} and it has been proved for several subclasses of trees \cite{ADOZ,aliste2014proper,loebl2018isomorphism,
martin2008distinguishing}. 

In a recent paper, Aliste-Prieto, De-Mier, Zamora and the second author \cite{ADOZ} introduced the \emph{Deletion-Near-Contraction} relation (DNC relation), which leads to an algorithm for writing $\mathbf{X}_G$ in the star basis $\{\mathfrak{st}_\lambda: \lambda \vdash n\}$ introduced by Cho and van Willigenburg \cite{chromBases}.  For $T$ a tree we have:  
 \[\mathbf{X}_T = \sum_{\lambda\vdash n} c_\lambda \mathfrak{st}_\lambda,\]
 where $\lambda\vdash n$ denotes that $\lambda$ is a partition of $n$.

  The results in this abstract are as follows: (1) we give formulas for some of the coefficients $c_\lambda$; (2) we determine the smallest partition $\lambda$ in lexicographic order such that $c_\lambda \neq 0$ and compute this coefficient, we call this partition the \emph{leading partition} of $\csft$; (3) we give a reconstruction algorithm for trees of diameter less than 6. We believe these results lay the groundwork for an inductive reconstruction algorithm that could hopefully prove the Tree Isomorphism Conjecture. Our approach is novel in that it provides an explicit reconstruction directly from $\mathbf{X}_T$. Additionally, we present a new, much shorter proof for the diameter-five case, different from the proof in \cite{GOT-CSF}.

This paper is organized as follows: In Section 2, we review basic definitions. In Section 3, we review the deletion-near-contraction relation, which allows us to efficiently compute the CSF in the star basis. In addition, we present some results about the coefficients that occur. In Section 4, we present our main result about the leading partition and its coefficient. In Section 5, we prove that the Tree Isomorphism Conjecture is true for all trees with a diameter of less than 6.

    \section{Graphs and the Chromatic Symmetric Function} For basic graph theory background see \cite{westBook}.  Here we give some basic definitions and set the notation.  We are interested in finite, simple graphs. A graph $G = (V,E)$ is defined via $V$, the set of vertices, and $E$ the set of edges. We say that $|V|$ is the \defn{order} of $G$. In this abstract we further restrict to the case when $G$ is a \defn{tree}, an acyclic, connected graph. A \defn{forest} is the disjoint union of trees. Every edge is determined by its endpoints,  $e = uv$, where $u,v\in V$. The \defn{degree} of a vertex $v$ is the number of edges having $v$ as an endpoint.  A \defn{leaf} is a vertex of degree 1, and a \defn{leaf edge} is an edge incident to a leaf.  
    
    In this abstract we need a few more non-standard definitions.  An edge $e=uv$ is called an \defn{internal edge} if $u$ and $v$ have degree $\geq 2$. We use \defn{$I(T)$} for the set of internal edges of $T$. An \defn{internal vertex} is a vertex that is not a leaf. An internal vertex is called \defn{deep} if it is not the endpoint of a leaf edge.  In the example below, $u$ and $w$ are internal vertices, and $u$ is deep, $e$ is an internal edge and $\ell$ is a leaf edge and $v$ is a leaf.  

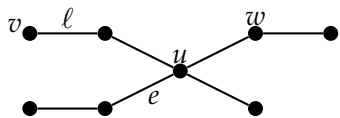
\centerline{
        \begin{tikzpicture}[auto=center,every node/.style={circle, fill=black, scale=0.5}]
            \node (a) at (1,0) {};
            \node (b) at (0,.5) {};
            \node (c) at (0,-.5) {};
            \node (d) at (-1,.5) {};
            \node (e) at (-1,-.5) {};
            \node (f) at (2, .5) {};
            \node (g) at (2, -.5) {};
            \node (h) at (3, .5) {};

            \draw[thick] (a)--(b);
            \draw[thick] (a)--(c);
            \draw[thick] (b)--(d);
            \draw[thick] (c)--(e);
            \draw[thick] (f)--(h);
            \draw[thick] (g)--(a)--(f);

            \node[fill=none, scale=1.75] at (-.5, .7) {$\ell$};
            \node[fill=none, scale=1.75] at (-1.2, .6) {$v$};
            \node[fill=none, scale=1.75] at (1, 0.2) {$u$};
            \node[fill=none, scale=1.75] at (2, 0.7) {$w$};
            \node[fill=none, scale=1.75] at (.65, -0.35) {$e$};
        \end{tikzpicture}
}

A \defn{proper} coloring of a graph $G$ with vertex set $\{v_1, \ldots, v_n\}$ is a function $f:V\rightarrow \mathbb{Z}_{\geq 0}$, such that if $e = uv \in E$, then $f(u)\neq f(v)$. In \cite{stanley95symmetric}, Stanley introduced the \defn{chromatic symmetric function}:
\[\mathbf{X}_G = \sum_{f:V\rightarrow \mathbb{Z}_{>0}}x_{f(v_1)}\cdots x_{f(v_n)}\]
where the sum is over all proper colorings of $G$ and $x_1, x_2 \ldots$ are commuting variables. $\mathbf{X}_G$ is a homogeneous, symmetric function of degree $n$, the number of vertices. Let $\mathsf{Sym}_n$ be the vector space of homogeneous, symmetric functions of degree $n$. $\mathsf{Sym}_n$ has dimension \defn{$p(n)$} which is equal to the number of partitions of $n$. Recall that a \defn{partition} of $n$, $\lambda=(\lambda_1, \ldots, \lambda_k)$, is a multiset of positive integers written in  weakly decreasing order and $\sum_{i=1}^k\lambda_i= n$, denoted by $\lambda\vdash n$. The \defn{length} of $\lambda$ is $\ell(\lambda)=k$. Let $p_r = x_1^r+ x_2^r+ \cdots$ for any $r\in \mathbb{Z}_{>0}$ and $p_\lambda = p_{\lambda_1}\cdots p_{\lambda_k}$. The \defn{power sum} basis $\{p_\lambda\, |\, \lambda\vdash n\}$ is a linear basis for $\mathsf{Sym}_n$. In \cite{stanley95symmetric} Theorem 2.5, Stanley wrote $\mathbf{X}_G$ in the power basis, we do not include this here as we do not need it. But we give the power expansion for the \defn{star graph} $St_{n}$, the tree with one vertex of degree $n-1$ and $n-1$ vertices of degree 1.
\begin{equation*}
    \label{eq:star_p_exp}
\mathfrak{st}_{n}:= \mathbf{X}_{St_n} = \sum_{r=0}^{n-1}(-1)^r\binom{n-1}{r}p_{(r+1,1^{n-r-1})}
\end{equation*}
More generally, if $G$ is a forest of stars, we write $G =  St_{\lambda_1}\sqcup \cdots \sqcup St_{\lambda_k}$, where $\sqcup$ is the disjoint union of graphs. We write $St_{\lambda}$ for the forest of stars with connected components of orders $\lambda_i$, for all $i$. We define \defn{$\mathfrak{st}_\lambda := \mathbf{X}_{St_\lambda}= \stfrak_{\lambda_1}\cdots \stfrak_{\lambda_k}$}.  Here is $St_{9,2,1}$

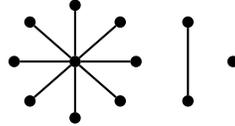
\centerline{
\begin{tikzpicture}
[scale = 0.3,thick, baseline={(0,-1ex/2)}]
\tikzstyle{vertex} = [black, shape = circle, fill=black, minimum size = 3.5pt, inner sep = 1pt]
\node[vertex] (G7) at (12.0, -1.75) [shape = circle, draw] {};
\node[vertex] (G6) at (11.3, 0) [shape = circle, draw] {};
\node[vertex] (G5) at (12.0, 1.75) [shape = circle, draw] {};
\node[vertex] (G4) at (14.0, 2.5) [shape = circle, draw] {};
\node[vertex] (G3) at (16.0, 1.75) [shape = circle, draw] {};
\node[vertex] (G8) at (16.0, -1.75) [shape = circle, draw] {};
\node[vertex] (G2) at (14.0, -2.5) [shape = circle, draw] {};
\node[vertex] (G1) at (16.7, 0) [shape = circle, draw] {};
\node[vertex] (G0) at (14.0, 0) [shape = circle, draw] {};
\draw[thick] (G0) -- (G1);
\draw[thick] (G0) -- (G2);
\draw[thick] (G0) -- (G3);
\draw[thick] (G0) -- (G4);
\draw[thick] (G0) -- (G5);
\draw[thick] (G0) -- (G6);
\draw[thick] (G0) -- (G7);
\draw[thick] (G0) -- (G8);
\node[vertex] (G10) at (19,-1.75) [shape=circle,draw ] {};
\node[vertex] (G11) at (19,1.75) [shape=circle,draw ] {};
\node[vertex] (G12) at (21,0) [shape=circle,draw ] {};
\draw[thick] (G10) -- (G11);
\end{tikzpicture}}

\label{fig:st9}

Cho and van Willigenburg showed \cite{chromBases} that $\{\mathfrak{st}_\lambda \, |\, \lambda \vdash n\}$ is a basis of $\mathsf{Sym}_n$ called the \defn{star basis}. In the following section, we give a relation and an algorithm for writing $\mathbf{X}_G$ in the star basis. 

    \section{Deletion-Near-Contraction}
    In this section, we review the \defn{deletion-near-contraction} (DNC) relation, recently introduced in \cite{ADOZ}. This relation is a modification of the classical deletion-contraction relation used to compute the chromatic polynomial of any graph.
\begin{itemize}
    \item {\bf Deletion:} This is the classic deletion of an edge in a graph. Given a graph $G$, we denote the resulting graph obtained by deleting an edge $e$ by $G \setminus e$.
 \item \textbf{Leaf-contraction:} Given a graph $G$ and an edge $e$ in $G$, the \defn{leaf-contracted graph}, $G\odot e$, is obtained by contracting $e$ and attaching a leaf $\ell_e$ to the vertex that results from the contraction of $e$. 
    \item \textbf{Dot-contraction:} Given an edge $e$, the \defn{dot-contracted graph}, $(G\odot e)\setminus \ell_e$, is obtained by contracting the edge $e$ and adding an isolated vertex, $v$, to the resulting graph. This can be formulated in terms of the leaf-contraction operation as simply removing the edge $\ell_e$. 
\end{itemize}

See Example \ref{exa:DNC-tree} for an illustration of these operations. In \cite{ADOZ}, the authors proved that the CSF satisfies a relation involving the three operations defined above. 
\begin{proposition}[\cite{ADOZ}](The deletion-near-contraction relation or DNC relation) \label{prop:dnc-formula}
    For a simple graph $G$ and any edge $e$ in $G$, we have
    \[
    \mathbf{X}_G = \mathbf{X}_{G \setminus e} - \mathbf{X}_{(G \odot e)\setminus \ell_e} + \mathbf{X}_{G \odot e}
    \]
\end{proposition}
If $e$ is a leaf edge, $\mathbf{X}_G = \mathbf{X}_{G \odot e}$ and $\mathbf{X}_{G \setminus e} = \mathbf{X}_{(G \odot e)\setminus \ell_e}$. Hence, we only apply it to internal edges.  In \cite{ADOZ}, the authors show that we can recursively apply the DNC relation on internal edges until $\mathbf{X}_G$ can be written as a linear combination of $\mathbf{X}_H$, where $H$ is a forest of star graphs. Stars do not have internal edges. This process is formalized in the Star-Expansion Algorithm presented in \cite{ADOZ}. We call an output of the star-expansion algorithm a \defn{DNC tree} for $G$, \defn{$\mathcal{T}(G)$}. Every time we apply the DNC relation we obtain forests with fewer internal edges. This leads us to an algorithm that gives the expansion of $\mathbf{X}_G$ in the star basis. For details about this algorithm see \cite{GOT-CSF} Section 3. 

\begin{example}\label{exa:DNC-tree}
    Figure \ref{fig:dnc-tree} shows an example of how to apply the star-expansion algorithm. We use red to indicate the internal edge on which we apply the DNC relation. In particular, it says that for the graph $T$ at the root, we have \[\csft = 
    -\mathfrak{st}_{(4,2,1)}+ \mathfrak{st}_{(4,3)}+\mathfrak{st}_{(5,1,1)}+\mathfrak{st}_{(5,2)}-2\mathfrak{st}_{(6,1)}+\mathfrak{st}_{(7)}~.\]
    
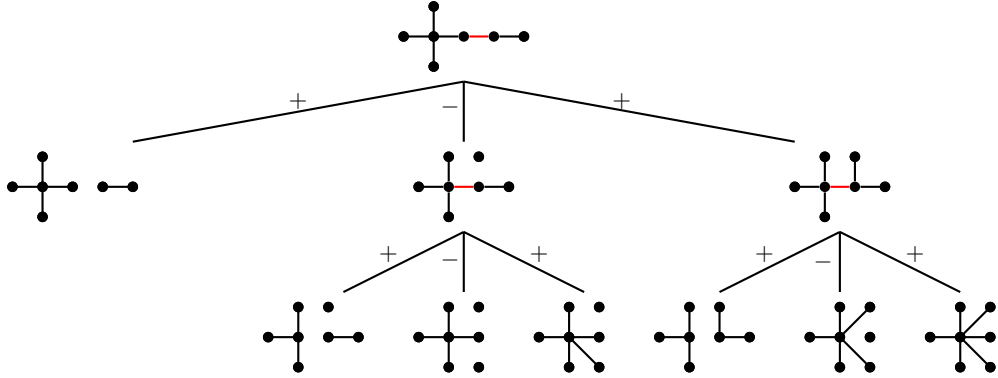
\centering
    \begin{tikzpicture}[auto=center,every node/.style={circle, fill=black, scale=0.35}, style=thick, scale=0.4] \label{tikz:dnc-tree}
    \filldraw[black] (0, 0) coordinate (A1) circle (4pt) node{};
    \filldraw[black] (1, 0) coordinate (A2) circle (4pt) node{};
    \filldraw[black] (1, 1) coordinate (A3) circle (4pt) node{};
    \filldraw[black] (1, -1) coordinate (A4) circle (4pt) node{};
    \node (A5) at (2,0) {};
    \node (A6) at (3,0) {};
    \filldraw[black] (4, 0) coordinate (A7) circle (4pt) node{};
    
    \draw(A1) -- (A2);
    \draw(A2) -- (A3);
    \draw(A2) -- (A4);
    \draw(A2) -- (A5);
    \draw[color=red](A5) -- (A6);
    \draw(A6) -- (A7);

    \filldraw[black] (-13, -5) coordinate (A1) circle (4pt) node{};
    \filldraw[black] (-12, -5) coordinate (A2) circle (4pt) node{};
    \filldraw[black] (-12, -4) coordinate (A3) circle (4pt) node{};
    \filldraw[black] (-12, -6) coordinate (A4) circle (4pt) node{};
    \filldraw[black] (-11, -5) coordinate (A5) circle (4pt) node{};
    \filldraw[black] (-10, -5) coordinate (A6) circle (4pt) node{};
    \filldraw[black] (-9, -5) coordinate (A7) circle (4pt) node{};
    
    \draw(A1) -- (A2);
    \draw(A2) -- (A3);
    \draw(A2) -- (A4);
    \draw(A2) -- (A5);
    \draw(A6) -- (A7);
    
    \filldraw[black] (0.5, -5) coordinate (A1) circle (4pt) node{};
    \node (A2) at (1.5,-5) {};
    \filldraw[black] (1.5, -4) coordinate (A3) circle (4pt) node{};
    \filldraw[black] (1.5, -6) coordinate (A4) circle (4pt) node{};
    \node (A5) at (2.5,-5) {};
    \filldraw[black] (3.5, -5) coordinate (A6) circle (4pt) node{};
    \filldraw[black] (2.5, -4) coordinate (A7) circle (4pt) node{};
    
    \draw(A1) -- (A2);
    \draw(A2) -- (A3);
    \draw(A2) -- (A4);
    \draw[red](A2) -- (A5);
    \draw(A5) -- (A6);

    \filldraw[black] (13, -5) coordinate (A1) circle (4pt) node{};
    \node (A2) at (14,-5) {};
    \filldraw[black] (14, -4) coordinate (A3) circle (4pt) node{};
    \filldraw[black] (14, -6) coordinate (A4) circle (4pt) node{};
    \node (A5) at (15,-5) {};
    \filldraw[black] (16, -5) coordinate (A6) circle (4pt) node{};
    \filldraw[black] (15, -4) coordinate (A7) circle (4pt) node{};

    \draw(A1) -- (A2);
    \draw(A2) -- (A3);
    \draw(A2) -- (A4);
    \draw[red](A2) -- (A5);
    \draw(A5) -- (A7);
    \draw(A5) -- (A6);

    \filldraw[black] (-4.5, -10) coordinate (A1) circle (4pt) node{};
    \filldraw[black] (-3.5, -10) coordinate (A2) circle (4pt) node{};
    \filldraw[black] (-3.5, -9) coordinate (A3) circle (4pt) node{};
    \filldraw[black] (-3.5, -11) coordinate (A4) circle (4pt) node{};
    \filldraw[black] (-2.5, -10) coordinate (A5) circle (4pt) node{};
    \filldraw[black] (-1.5, -10) coordinate (A6) circle (4pt) node{};
    \filldraw[black] (-2.5, -9) coordinate (A7) circle (4pt) node{};
    
    \draw(A1) -- (A2);
    \draw(A2) -- (A3);
    \draw(A2) -- (A4);
    \draw(A5) -- (A6);
    
    \filldraw[black] (0.5, -10) coordinate (A1) circle (4pt) node{};
    \filldraw[black] (1.5, -10) coordinate (A2) circle (4pt) node{};
    \filldraw[black] (1.5, -9) coordinate (A3) circle (4pt) node{};
    \filldraw[black] (1.5, -11) coordinate (A4) circle (4pt) node{};
    \filldraw[black] (2.5, -10) coordinate (A5) circle (4pt) node{};
    \filldraw[black] (2.5, -11) coordinate (A6) circle (4pt) node{};
    \filldraw[black] (2.5, -9) coordinate (A7) circle (4pt) node{};
    
    \draw(A1) -- (A2);
    \draw(A2) -- (A3);
    \draw(A2) -- (A4);
    \draw(A2) -- (A5);
    
    \filldraw[black] (4.5, -10) coordinate (A1) circle (4pt) node{};
    \filldraw[black] (5.5, -10) coordinate (A2) circle (4pt) node{};
    \filldraw[black] (5.5, -9) coordinate (A3) circle (4pt) node{};
    \filldraw[black] (5.5, -11) coordinate (A4) circle (4pt) node{};
    \filldraw[black] (6.5, -10) coordinate (A5) circle (4pt) node{};
    \filldraw[black] (6.5, -11) coordinate (A6) circle (4pt) node{};
    \filldraw[black] (6.5, -9) coordinate (A7) circle (4pt) node{};
    
    \draw(A1) -- (A2);
    \draw(A2) -- (A3);
    \draw(A2) -- (A4);
    \draw(A2) -- (A5);
    \draw(A2) -- (A6);
    
    \filldraw[black] (8.5, -10) coordinate (A1) circle (4pt) node{};
    \filldraw[black] (9.5, -10) coordinate (A2) circle (4pt) node{};
    \filldraw[black] (9.5, -9) coordinate (A3) circle (4pt) node{};
    \filldraw[black] (9.5, -11) coordinate (A4) circle (4pt) node{};
    \filldraw[black] (10.5, -10) coordinate (A5) circle (4pt) node{};
    \filldraw[black] (11.5, -10) coordinate (A6) circle (4pt) node{};
    \filldraw[black] (10.5, -9) coordinate (A7) circle (4pt) node{};
    
    \draw(A1) -- (A2);
    \draw(A2) -- (A3);
    \draw(A2) -- (A4);
    \draw(A5) -- (A7);
    \draw(A5) -- (A6);
    
    \filldraw[black] (13.5, -10) coordinate (A1) circle (4pt) node{};
    \filldraw[black] (14.5, -10) coordinate (A2) circle (4pt) node{};
    \filldraw[black] (14.5, -9) coordinate (A3) circle (4pt) node{};
    \filldraw[black] (14.5, -11) coordinate (A4) circle (4pt) node{};
    \filldraw[black] (15.5, -10) coordinate (A5) circle (4pt) node{};
    \filldraw[black] (15.5, -11) coordinate (A6) circle (4pt) node{};
    \filldraw[black] (15.5, -9) coordinate (A7) circle (4pt) node{};
    
    \draw(A1) -- (A2);
    \draw(A2) -- (A3);
    \draw(A2) -- (A4);
    \draw(A2) -- (A6);
    \draw(A2) -- (A7);
    
    \filldraw[black] (17.5, -10) coordinate (A1) circle (4pt) node{};
    \filldraw[black] (18.5, -10) coordinate (A2) circle (4pt) node{};
    \filldraw[black] (18.5, -9) coordinate (A3) circle (4pt) node{};
    \filldraw[black] (18.5, -11) coordinate (A4) circle (4pt) node{};
    \filldraw[black] (19.5, -10) coordinate (A5) circle (4pt) node{};
    \filldraw[black] (19.5, -11) coordinate (A6) circle (4pt) node{};
    \filldraw[black] (19.5, -9) coordinate (A7) circle (4pt) node{};
    
    \draw(A1) -- (A2);
    \draw(A2) -- (A3);
    \draw(A2) -- (A4);
    \draw(A2) -- (A5);
    \draw(A2) -- (A6);
    \draw(A2) -- (A7);

    \draw(2,-1.5) -- (-9,-3.5);
    
    \node[fill=none] at (-3.5, -2.15) {\huge $+$};
    \draw (2,-1.5) -- (2,-3.5);
    \node[fill=none] at (1.55, -2.35) {\huge $-$};
    \draw (2,-1.5) -- (13,-3.5);
    \node[fill=none] at (7.25, -2.15) {\huge $+$};
    
    \draw [fill=none](2,-6.5) -- (-2,-8.5);

    \node[fill=none] at (-0.5, -7.25) {\huge $+$};
    
    \draw(2,-6.5) -- (2,-8.5);
    
    \node[fill=none] at (1.55, -7.45) {\huge $-$};
    
    \draw (2,-6.5) -- (6,-8.5);

    \node[fill=none] at (4.5, -7.25) {\huge $+$};

    \draw (14.5,-6.5) -- (10.5,-8.5);
    
    \node[fill=none] at (12, -7.25) {\huge $+$};
    
    \draw (14.5,-6.5) -- (14.5,-8.5);

    \node[fill=none] at (13.95, -7.5) {\huge $-$};
    
    \draw (14.5,-6.5) -- (18.5,-8.5);

    \node[fill=none] at (17, -7.25) {\huge $+$};
    \end{tikzpicture}
    \label{fig:dnc-tree}

\end{example}
\begin{remark}\label{remark:paths-coefficients}
    If $G$ has $n$ vertices and $\mathbf{X}_G = \sum_{\lambda\vdash n} c_\lambda \stfrak_\lambda$ is written in the star basis, then 
    $c_\lambda = (-1)^{m} |\mathcal{S}_\lambda|,$
    where $\mathcal{S}_\lambda$ is the set of paths in a DNC tree $\mathcal{T}(G)$ that end in a star forest $H$ whose connected components' orders are given by the parts in $\lambda$, and $m$ is the number of dot-contractions performed throughout any of these paths. 
\end{remark}

Remark \ref{remark:paths-coefficients} allows us to compute coefficients using combinatorial arguments for the number of paths in a DNC tree. For example, in the case of hook partitions.  
\begin{proposition}\cite[Proposition 3.8]{GOT-CSF}\label{prop:hook-coeff}
   Let $T$ be a tree of order $n$ and $I(T)$ the set of internal edges of $T$. If 
     $\csft = \sum_{\lambda\vdash n} c_\lambda \mathfrak{st}_\lambda$, then 
    $c_{(n-m,1^m)} = (-1)^m \binom{\#I(T)}{m}.$ In particular, $|c_{(n-1,1)}|=\#I(T)$.
\end{proposition}

\section{The Leading Partition}
Let $F$ be a forest with $n$ vertices with internal edges $I(F)$. We call the connected components of $F \setminus I(F)$
the \defn{leaf components} of $F$, and we denote by $\lc(F)$ the partition of $n$ whose parts are the orders of these components, that is, 
\[
\lc(F)\coloneqq\lambda(F\setminus I(F)).
\]
We call $\lc(F)$ the \defn{leaf component partition} of $F$.
Notice that $F \setminus I(F)$ is a spanning subgraph, i.e., every vertex in $F$ is also a vertex in $F \setminus I(F)$. Furthermore, $F\setminus I(F)$ is a forest whose connected components are all stars. Thus, a leaf component is always a star tree. 
\begin{example}\label{lc}
   The tree $T$ from Example \ref{exa:DNC-tree} has leaf components $St_4$, $St_2$, and $St_1$; hence $\lc(T) = (4,2,1)$.
   
 
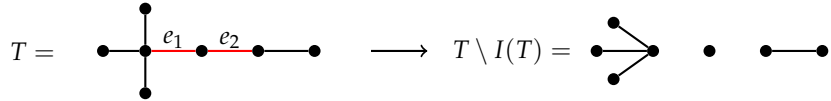
\centerline{       \begin{tikzpicture}[auto=center,every node/.style={circle, fill=black, scale=0.4}, scale=0.75]
            \node (a) at (0.25, 0) {};
            \node (b) at (1, 0) {};
            \node (c) at (1, .75) {};
            \node (d) at (1, -.75) {};
            \node (e) at (2, 0) {};
            \node (f) at (3, 0) {};
            \node (g) at (4, 0) {};
                
            \draw[thick] (a) -- (b);
            \draw[thick] (b) -- (c);
            \draw[thick] (b) -- (d);
            \draw[thick][color=red, thick] (b) -- (e);
            \draw[thick][color=red, thick] (e) -- (f);
            \draw[thick] (f) -- (g);
                
            \node (h) at (9, 0) {};
            \node (i) at (10, 0) {};
            \node (j) at (9.3, .5) {};
            \node (k) at (9.3, -.5) {};
            \node (l) at (11, 0) {};
            \node (m) at (12, 0) {};
            \node (n) at (13, 0) {};
                
            \draw[thick] (h) -- (i);
            \draw[thick] (i) -- (j);
            \draw[thick] (i) -- (k);
            \draw[thick] (m) -- (n);
                
            \draw[thick, ->] (5,0)--(6,0);
                
            \node[fill=none, scale=2] (p) at (-1,0) {$T = $};
            \node[fill=none, scale=2] (q) at (7.5,0) {$T\setminus I(T) =$};
            \node[fill=none, scale=2] at (1.5, 0.2) {$e_1$};
            \node[fill=none, scale=2] at (2.5, 0.2) {$e_2$};
        \end{tikzpicture}}
\end{example}
For $\lambda, \mu \vdash n$, we say that $\lambda \leq \mu$ if $\lambda = \mu$ or if $\lambda_i = \mu_i$ for $1\leq i<j$ and $\lambda_j<\mu_j$ for some $1\leq j \leq \ell(\lambda)$. This is called the \defn{lexicographic order} on the set of partitions, and it is a total order. 
In what follows, we assume that $\{\stfrak_\lambda : \lambda \vdash n\}$ is an ordered basis with respect to the lexicographic order. For any tree $T$ of order $n$, we write \[\mathbf{X}_T = \sum_{\lambda\vdash n} c_\lambda \mathfrak{st}_\lambda, \]
where the summands are listed in increasing lexicographic order.

    \begin{definition}
        Let $F$ be a forest of order $n$ with $\mathbf{X}_F=\sum_{\lambda \vdash n}c_\lambda \mathfrak{st}_\lambda$. The \defn{leading partition} of $\mathbf{X}_F$ is the smallest partition $\lambda \vdash n$, in lexicographic order, such that $c_{\lambda} \neq 0$. We then say that $c_{\lambda}$ is the \defn{leading coefficient}. We denote the leading partition of $F$ by \defn{$\lead(\csff)$}.
    \end{definition}

\begin{example}\label{lead}
    For the tree $T$ from Example \ref{exa:DNC-tree}, the leading partition is $\lead(\csft)=(4,2,1)$, the leading coefficient is $c_{(4,2,1)} = -1$. Observe that $\lc(T) = \lead(\csft)$.
\end{example}

In Examples \ref{lc} and \ref{lead}, we see that $\lead = \lc$. In Section 4 of \cite{GOT-CSF}, we prove that this is always the case. We also give an elegant combinatorial formula for the leading coefficient. Recall that a deep vertex is an internal vertex that is not an endpoint of a leaf edge. The number of deep vertices is equal to the number of 1s in $\lc$.

\begin{theorem}\cite[Theorem 4.16 and Theorem 4.29]{GOT-CSF}\label{Thm:lead}
    Let $F$ be a forest with $n$ vertices. Then $ \lead(\csff) = \lc(F)$.
    If in addition, $F$ has deep vertices $u_1, \ldots, u_m$, then:
    \[
    c_{\lead} = (-1)^m \prod_{i=1}^m (\deg(u_i)-1)
    \]
\end{theorem}

This result is far from trivial. In particular, it is not even clear that there exists a path in a DNC tree $\mathcal{T}(F)$ from the root to a star forest indexed by $\lc(F)$ since the path obtained by applying $\#I(F)$ deletions is not always possible. For instance, in Example \ref{exa:DNC-tree}, there is no path obtained by performing two deletions.

\begin{example} Let $T$ be the tree below. Notice that $\lc(T)=(2,2,2,1)$

\centerline{
        \begin{tikzpicture}[auto=center,every node/.style={circle, fill=black, scale=0.4}, scale=0.8]
            \node (1a) at (0,0) {};
            \node (1b) at (1,0) {};
            \node (1c) at (2,0) {};
            \node (1d) at (3, .5) {};
            \node (1e) at (3,-.5) {};
            \node (1f) at (4, .5) {};
            \node (1g) at (4, -.5){};
            \node[scale=2.5, fill=none] at (1.9, .3) {$u$};
            \draw[thick] (1a) -- (1b) -- (1c);
            \draw[thick] (1c) -- (1d) -- (1f);
            \draw[thick] (1c) -- (1e) -- (1g);            
        \end{tikzpicture}}
   
\noindent We have $\csft = -2\mathfrak{st}_{(2^3,1)} +3\mathfrak{st}_{(3,2,1^2)}   +3\mathfrak{st}_{(3,2^2)}
    -\mathfrak{st}_{(4,1^3)}
    -6\mathfrak{st}_{(4,2,1)}+ 3\mathfrak{st}_{(5,1,1)} + 3\mathfrak{st}_{(5,2)} - 3\mathfrak{st}_{(6,1)} +
    \mathfrak{st}_{(7)}$.  We can see that $\lead(\csft)=(2,2,2,1) = \lc(T)$ and there is only one deep vertex $u$ of degree 3, hence the coefficient is $c_{\lead}=(-1)^1(3-1) = -2$ 
\end{example}

\subsection{Tree reconstructions using $\lead$}

Given any tree $T$, we can now determine the leading partition of $\csft$ based on the properties of the tree $T$ itself. In this subsection, we show that $\lead(\csft)$ allows us to make some progress towards a positive answer to the tree isomorphism conjecture. In particular, we immediately obtain the following corollary:

\begin{corollary} \label{cor:different-leading}
    If $T_1$ and $T_2$ are trees whose leaf components have different orders, then $\mathbf{X}_{T_1}\neq \mathbf{X}_{T_2}$, that is, if $\lc(T_1)\neq \lc(T_2)$, then $\mathbf{X}_{T_1}\neq \mathbf{X}_{T_2}$. \qed
\end{corollary}

The leading partition also allows us to positively answer the  conjecture for another infinite family of trees.
\begin{definition}
    A \defn{bi-star} is a tree consisting of two star graphs whose centers are joined by an internal edge. An \defn{extended bi-star} is a tree consisting of two star graphs whose centers are connected by a path of one or more deep vertices of degree 2.
\end{definition}

\begin{example}
    The extended bi-star shown below has leading partition $(6,4,1^4)$
    
\centerline{ 
\begin{tikzpicture}[every node/.style={circle, fill=black, scale=0.4},scale=0.8]
            \node (a) at (0,0) {};
            \node (b) at (1,0) {};
            \node (c) at (2,0) {}; 
            \node (d) at (3,0) {};
            \node (e) at (4,0) {};
            \node (f) at (5,0) {};

            \draw[thick] (a)--(b);
            \draw[thick] (b)--(c);
            \draw[thick] (c)--(d);
            \draw[thick] (d)--(e);
            \draw[thick] (e)--(f);

            \node (a1) at (-0.6,.7) {};
            \node (a2) at (-.5,-.7){};
            \node (a3) at (-.7, 0) {};
            \node (a4) at (-.7, .4) {};
            \node (a5) at (-.6, -.4) {};

            \draw[thick] (a)--(a1);
            \draw[thick] (a)--(a2);
            \draw[thick] (a)--(a3);
            \draw[thick] (a)--(a4);
            \draw[thick] (a)--(a5);

            \node (f1) at (5.7,0) {};
            \node (f2) at (5.7, .4) {};
            \node (f3) at (5.7, -.4){};

            \draw[thick] (f)--(f1);
            \draw[thick] (f)--(f2);
            \draw[thick] (f)--(f3);           
        \end{tikzpicture}}
\end{example}

\begin{corollary}\cite[Corollary 4.23 and Corollary 4.25]{GOT-CSF} \label{cor:bi-stars-leading}
    Let $T$ be a tree of order $n$. Then, $\lead(\csft)=(i,j,1^{n-i-j})$ for some $i,j>1$ if and only if $T$ is a bi-star or extended bi-star with leaf-stars $St_{i}$ and $St_{j}$ separated by $n-i-j$ deep vertices of degree 2. Therefore, bi-stars and extended bi-stars are distinguished by their CSF. 
\end{corollary}

Note that bi-stars and extended bi-stars are particular cases of caterpillars, which are already known to be distinguished by their CSF \cite{loebl2018isomorphism,martin2008distinguishing}. We included Corollary \ref{cor:bi-stars-leading} here to illustrate that Theorem \ref{Thm:lead} has strong consequences and also the proofs in \cite{loebl2018isomorphism, martin2008distinguishing} are not constructive and do not use $\mathbf{X}_T$ directly. 

In what follows, we show that edge adjacencies can also be recovered from $\csft$, which, together with $\lead$ will allow us to reconstruct other families of trees.  If $T$ is a tree and $\mathcal{L}_1$ and $\mathcal{L}_2$ are two leaf components with central vertices $v_1$ and $v_2$, respectively, we say that $\mathcal{L}_1$ and $\mathcal{L}_2$ are \defn{adjacent} if $v_1v_2 \in E(T)$. In addition, we will refer to $\mathcal{L}_1$ and $\mathcal{L}_2$ as the \defn{leaf component endpoints} for the internal edge $e=v_1v_2$.

\begin{example}
    In the tree below, the leaf components $\mathcal{L}_1$ (with center $v_1$) and  $\mathcal{L}_2$ (with center $v_2$), are adjacent and they are the leaf component endpoints of the edge $v_1v_2$.
    
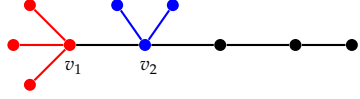
\centerline{   
        \begin{tikzpicture}[every node/.style={circle, fill=black, scale=0.4}, thick]
            \node[fill=red] (a) at (0,0) {};
            \node[fill=red] (a1) at (-.53, -.53) {};
            \node[fill=red] (a2) at (-.75, 0) {};
            \node[fill=red] (a3) at (-.53, .53) {};

            \node[scale=1.5, fill=none] at (0.05, -.3) {$v_1$};
            
            \node[fill=blue] (b) at (1,0) {};
            \node[fill=blue] (b1) at (0.63, 0.53) {};
            \node[fill=blue] (b2) at (1.37, 0.53) {};

            \node[scale=1.5, fill=none] at (1.05, -.3) {$v_2$};

            \node (c) at (2,0) {};
            
            \node (d) at (3,0) {};
            \node (d1) at (3.75, 0) {};

            \draw[red] (a) -- (a1);
            \draw[red] (a) -- (a2);
            \draw[red] (a) -- (a3);

            \draw[blue] (b) -- (b1);
            \draw[blue] (b) -- (b2);

            \draw(d)--(d1);

            \draw (a) -- (b);
            \draw(b) -- (c);
            \draw(c)--(d);
        \end{tikzpicture}}
\end{example}
For two partitions $\lambda$ and $\mu$, we use \defn{$\lambda-\mu$} to denote the \defn{multiset difference} between $\lambda$ and $\mu$, that is, for each $i \in \{1,\ldots,n\}$, the multiplicity of $i$ in $\lambda-\mu$ is $\max\{m_i(\lambda)-m_i(\mu), 0\}$, where $m_i(\lambda)$ is the multiplicity of $i$ in $\lambda$.
\begin{definition}
    Let $\csft=\sum_{\lambda \vdash n}c_\lambda \mathfrak{st}_\lambda$ be the CSF of a tree of order $n$ with leading partition $\lead$. For any partition $\mu \vdash n$ without 1s,  $\ell(\mu) = \ell(\lead(\csft)) - 1$, and $c_\mu \neq 0$, we define the \defn{adjacency multiset} by  $E_\mu = \lead - \mu$.
\end{definition}

\begin{example}\label{ex:adjacencies-improper}
    Consider the following CSF:
    \[
    \csft = \textcolor{blue}{-\stfrak_{(4,2,1)}} + \textcolor{red}{\stfrak_{(4,3)}} + \stfrak_{(5,1,1)} + \textcolor{red}{\stfrak_{(5,2)}} - 2\stfrak_{(6,1)} + \stfrak_{(7)}
    \]
    The leading term has been colored blue, and the terms indexed by partitions of length $\ell(\lead(\csft)) - 1$ without 1s have been colored red. For $\mu=(4,3)$, we obtain $E_\mu = \lmulti 2,1 \rmulti$ and for $\nu = (5,2)$, $E_{\nu} = \lmulti 4,1 \rmulti$. From Example \ref{lc}, we see that $E_\mu$ and $E_{\nu}$ contain the orders of leaf components that are adjacent in $T$, allowing us to find the internal edges $e_1$ and $e_2$ in that example from $\mathbf{X}_T$.
\end{example}

    The following proposition makes precise the relation between the $E_\mu$ and the edges with endpoints the centers of leaf components of a tree. For a part $i$ of a partition $\lambda$, we let $m_i$ denote its multiplicity.
\begin{proposition}\cite[Proposition 5.14]{GOT-CSF} \label{prop:k-edge-adjacencies}
    Assume that $\csft$ is the CSF of a tree of order $n$ with leading partition $\lead(\csft)=(n^{m_n},\ldots, 1^{m_1})$. Let $\mu \vdash n$ such that $c_\mu \neq 0$ in $\csft$, $\ell(\mu) = \ell(\lead(\csft))-1$, and $\mu$ contains no 1s. Then,
    \begin{enumerate}
        \item[(a)] If $m_1=0$, then $E_\mu = \lmulti p,q \rmulti$, where $p$ and $q$ are orders of two adjacent leaf components in $T$. And $c_\mu$ is the number of internal edges with leaf component endpoints of orders $p$ and $q$.
        \item[(b)] If $m_1=1$, then $E_\mu = \lmulti 1, q \rmulti$, where $q$ is the order of a leaf component adjacent to the deep vertex. Then, $c_\mu$ is the number of leaf components of order $q$ adjacent to the deep vertex.
    \end{enumerate}
\end{proposition}

\begin{example}
    Example \ref{ex:adjacencies-improper} illustrates (b). For case (a), consider the following CSF:
    \begin{eqnarray*}
        \csft &=& \textcolor{blue}{\stfrak_{(3^2,2^2)}}-2\stfrak_{(4,3,2,1)}-\stfrak_{(5,2^2,1)}+\stfrak_{(5,3,1^2)}+\textcolor{red}{2\stfrak_{(5,3,2)}} + 2\stfrak_{(6,2,1^2)} + \textcolor{red}{\stfrak_{(6,2,2)}}\\
        &&-2\stfrak_{(6,3,1)}-\stfrak_{(7,1^3)}-4\stfrak_{(7,2,1)}+\stfrak_{(7,3)}+3\stfrak_{(8,1^2)}+2\stfrak_{(8,2)}-3\stfrak_{(9,1)}+\stfrak_{(10)}
    \end{eqnarray*}
    $\lead = (3^2,2^2)$ has length 4, the terms indexed by partitions of length 3, $\mu = (5,3,2)$ and $\nu = (6,2,2)$, are highlighted in red. Thus, $E_\mu = \lmulti 3,2\rmulti$ and $E_{\nu} = \lmulti 3,3\rmulti$. By Proposition \ref{prop:k-edge-adjacencies}, there are $c_{(5,3,2)} = 2$ internal edges in $T$ with leaf component endpoints of orders 3 and 2, and there is $c_{(6,2,2)}=1$ internal edge with leaf component endpoints of orders 3. 

\end{example}

\begin{corollary}\label{cor:proper-tree-distinct-parts}
    Let $\csft$ be the CSF of a tree $T$ such that $\lead=\lead(\csft) = (\lambda_1, \ldots, \lambda_k)$ contains no 1s and has all distinct parts. Then, $T$ can be reconstructed from $\csft$. In particular, $T$ can be reconstructed from $\lead$ and the adjacency multisets.
\end{corollary}
\begin{example}\label{ex:proper-distinct-parts}
    For the sake of brevity, we only provide $\lead(\csft)$ and the coefficients of the partitions $\mu$ without 1s such that $\ell(\mu)=\ell(\lead(\csft))-1$. Consider $\csft$ with $\lead(\csft) = (9,7,6,5,4,3,2)$ and with the following coefficients indexed by partitions without 1s with $\ell(\lead(\csft))-1$ on the table on the left. From these data, we can reconstruct the tree $T$ on the right (edge adjacencies are colored red). 
    
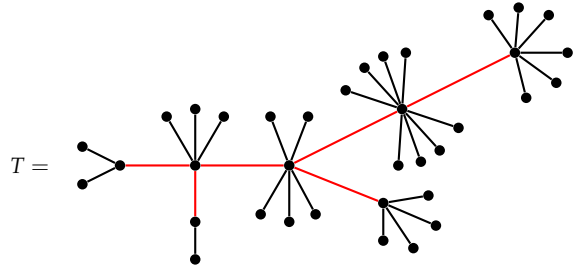
\begin{minipage}{.4\textwidth}
\small{
\begin{tabular}{||c |c |c||} 
         \hline
         $\mu$ & $c_\mu$ & $E_\mu$ \\ [0.5ex] 
         \hline\hline
         $(16,6,5,4,3,2)$ & $1$ & $\lmulti 9,7 \rmulti$ \\ 
         \hline
         $(15,7,5,4,3,2)$ & $1$ & $\lmulti 9,6 \rmulti$ \\
         \hline
         $(11,9,7,4,3,2)$ & $1$ & $\lmulti 6,5 \rmulti$ \\
         \hline
         $(10,9,7,5,3,2)$ & $1$ & $\lmulti 6,4\rmulti$ \\
         \hline
         $(9,7,7,6,5,2)$ & $1$ &  $\lmulti 4,3 \rmulti$\\ 
         \hline
         $(9,7,6,6,5,3)$ & $1$ & $\lmulti 4,2 \rmulti$ \\
         \hline
    \end{tabular}}
\end{minipage}
\begin{minipage}{.4\textwidth}
    \begin{tikzpicture}[auto=center,thick, every node/.style={circle, fill=black, scale=0.35}]
            \node (c) at (0,0) {};
            \node (c1) at (-.38, 0.65) {};
            \node (c2) at (0, .75) {};
            \node (c3) at (.38, .65) {};
            \node (t) at (-1, 0) {};
            \node (t1) at (-1.5, 0.25) {};
            \node (t2) at (-1.5, -.25) {};
            \node (d) at (0, -.75) {};
            \node (d1) at (0, -1.25) {};
            \node (s) at (1.25,0) {};
            \node (s1) at (1, .65) {};
            \node (s2) at (1.5, .65) {};
            \node (s3) at (.87, -.65) {};
            \node (s4) at (1.6, -.65) {};
            \node (s5) at (1.25, -.75) {};
            \node (n) at (2.75, .75) {};
            \node (n1) at (2.8, 1.5) {};
            \node (n2) at (2.7, 0) {};
            \node (n3) at (2.5, 1.45){};
            \node (n4) at (2.25, 1.3) {};
            \node (n5) at (2, 1) {};
            \node (n6) at (3, 0.05) {};
            \node (n7) at (3.25, 0.2) {};
            \node (n8) at (3.5, 0.5){};
            \node (l) at (4.25, 1.5) {};
            \node (l1) at (4.8, 1.1) {};
            \node (l2) at (4.3, 2.1) {};
            \node (l3) at (4.7, 2) {};
            \node (l4) at (4.95, 1.5) {};
            \node (l5) at (3.9, 2) {};
            \node (l6) at (4.4, .9) {};
            \node (f) at (2.5, -.5) {};
            \node (f1) at (3.2, -.8) {};
            \node (f2) at (2.9, -1.1) {};
            \node (f3) at (3.1, -.4) {};
            \node (f4) at (2.5, -1) {};

            \draw (l1)--(l)--(l2);
            \draw (l3)--(l)--(l4);
            \draw (l5)--(l)--(l6);
            \draw[thick,red] (n) -- (l);
            \draw (c1)--(c)--(c2);
            \draw (c)--(c3);
            \draw (t1)--(t)--(t2);
            \draw (d)--(d1);
            \draw[thick,red] (t)--(c)--(d);
            \draw (s1)--(s)--(s2);
            \draw (s3)--(s)--(s4);
            \draw (s)--(s5);
            \draw[thick,red] (c)--(s);
            \draw (f1)--(f)--(f2);
            \draw (f3)--(f)--(f4);
            \draw[thick,red] (f) -- (s);
            \draw[thick,red] (s) -- (n);
            \draw (n1)--(n)--(n2);
            \draw (n3)--(n)--(n4);
            \draw (n6)--(n)--(n5);
            \draw (n7)--(n)--(n8);
            \node[fill=none, scale=2] at (-2.2,0) {$T=$};
        \end{tikzpicture}
\end{minipage}
\end{example}
    
\section{Trees of diameter $\leq 5$}

In this section, we show that trees of diameter at most five can be reconstructed from their CSF.  We give an explicit reconstruction algorithm 
using the leading partition which for these trees can have at most two 1s, the \emph{internal subgraph}, defined below, and 
Proposition \ref{prop:k-edge-adjacencies} on the \emph{adjacencies} between the leaf components of the tree.  In \cite{martin2008distinguishing}, the authors proved that one can compute the diameter of a tree from its CSF.  In \cite{ADOZ}, the authors showed that a subclass of trees of diameter at most five are distinguished from their CSF, providing a reconstruction algorithm. Here we show that \emph{all} trees of diameter at most five can be reconstructed from $\mathbf{X}_T$.

Trees of diameter $\leq 2$ are stars and there is only one such tree for a given number of vertices $k$, namely $St_k$. If $T$ has diameter 3, then $T$ is a bi-star with $\lead(\mathbf{X}_T) = (i,j)$, where $i$ and $j$ are the orders of its two leaf components, see Corollary \ref{cor:bi-stars-leading}. Therefore, in the remainder of this section we focus on trees with diameter 4 and 5.

Recall that a leaf component of a tree $T$ is a connected component of $T \setminus I(T)$. Define the \defn{internal degree} of a vertex $v$ as the number of internal edges having $v$ as an endpoint.
\begin{definition}
    Let $T$ be a tree and let $\{v_1, \ldots, v_l\}$ be the set of vertices of $T$ with internal degree strictly greater than 1. Let $L_i$ be the set of leaf-vertices that are incident to any $v_i$ for $1 \leq i \leq l$. Then, the \defn{internal subgraph} of $T$, $\mathcal{I}_T$, is the subgraph of $T$, with vertices $V(\mathcal{I}_T) = \{v_1, \ldots, v_l\} \cup L_1 \cup \cdots \cup L_l$ and all edges in $T$ with endpoints in $V(\mathcal{I}_T)$.
\end{definition}

\begin{example} A tree $T$ and its internal subgraph $\mathcal{I}_T$. The vertices in $T$ with internal degree greater than 1 are labeled $v_1,v_2$ and $v_3$. Their internal degrees are 3, 3, and 5, respectively.
\vskip 0in 
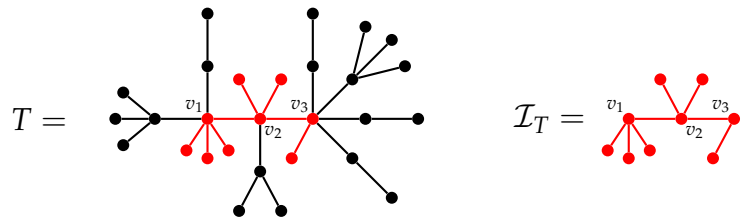
\centerline{    \begin{tikzpicture}[every node/.style={circle, fill=black, scale=0.4}, scale=0.7]    
        \node (a2) at (2, 3) {};
        \node (a4) at (4, 3) {};
        \node (a5) at (5, 2.75) {};
        \node (a52) at (5.5, 2.5) {};
        \node (a53) at (5.75, 2) {};
    
        \node (b0) at (0.4, 1.5) {};
        \node (b2) at (2,2) {};
        \node[fill=red] (b3) at (2.6, 1.75) {};
        \node[fill=red] (b32) at (3.4, 1.75) {};
        \node (b4) at (4, 2) {};
        \node (b5) at (4.75, 1.75) {};
        
        \node (c0) at (0.25,1) {};
        \node (c1) at (1,1) {};
        \node[fill=red] (c2) at (2,1) {};
        \node[fill=none, scale=1.5] at (1.75, 1.25) {$v_1$};
        \node[fill=red] (c3) at (3,1) {};
        \node[fill=none, scale=1.5] at (3.25, .75) {$v_2$};
        \node[fill=red] (c4) at (4,1) {};
        \node[fill=none, scale=1.5] at (3.75, 1.25) {$v_3$};
        \node (c5) at (5,1) {};
        \node (c6) at (6,1) {};
        
        \node (d0) at (0.4, 0.5) {};
        \node[fill=red] (d2) at (1.6, 0.4) {};
        \node[fill=red] (d22) at (2, 0.25) {};
        \node[fill=red] (d23) at (2.4, 0.4) {};
        \node (d3) at (3, 0) {};
        \node[fill=red] (d4) at (3.6, 0.25) {};
        \node (d5) at (4.75, 0.25) {};
        
        \node (e3) at (2.6, -0.75) {};
        \node (e32) at (3.4, -0.75) {};
        \node (e5) at (5.5, -0.5) {};
        
        \draw[thick] (c1) -- (c2);
        \draw[color=red, thick] (c3) -- (b3);
        \draw[color=red, thick] (c3) -- (b32);
        \draw[color=red, thick] (c2) -- (c3);
        \draw[color=red, thick] (c3) -- (c4);
        \draw[thick] (c4) -- (c5);
        \draw[thick] (c5) -- (c6);
        
        \draw[thick] (c2) -- (b2);
        \draw[thick] (b2) -- (a2);
        \draw[thick] (c4) -- (b4);
        \draw[thick] (b4) -- (a4);
        \draw[thick] (c0) -- (c1);
        \draw[thick] (c1) -- (b0);
        \draw[thick] (c1) -- (d0);
        \draw[color=red, thick] (c2) -- (d2);
        \draw[color=red, thick] (c2) -- (d22);
        \draw[color=red, thick] (c2) -- (d23);
        \draw[thick] (c3) -- (d3);
        \draw[thick] (d3) -- (e3);
        \draw[thick] (d3) -- (e32);
        \draw[color=red, thick] (c4) -- (d4);
        \draw[thick] (c4) -- (b5);
        \draw[thick] (c4) -- (d5);
        \draw[thick] (d5) -- (e5);
        \draw[thick] (b5) -- (a5);
        \draw[thick] (b5) -- (a52);
        \draw[thick] (b5) -- (a53);

    \node[fill=red] (bb3) at (10.6, 1.75) {};
    \node[fill=red] (bb32) at (11.4, 1.75) {};
    \node[fill=red] (cc2) at (10,1) {};
        \node[fill=none, scale=1.5] at (9.75, 1.25) {$v_1$};
        \node[fill=red] (cc3) at (11,1) {};
        \node[fill=none, scale=1.5] at (11.25, .75) {$v_2$};
        \node[fill=red] (cc4) at (12,1) {};
        \node[fill=none, scale=1.5] at (11.75, 1.25) {$v_3$};
        \node[fill=red] (dd2) at (9.6, 0.4) {};
        \node[fill=red] (dd22) at (10, 0.25) {};
        \node[fill=red] (dd23) at (10.4, 0.4) {};
        \node[fill=red] (dd4) at (11.6, 0.25) {};
        \draw[color=red, thick] (cc3) -- (bb3);
        \draw[color=red, thick] (cc3) -- (bb32);
        \draw[color=red, thick] (cc2) -- (cc3);
        \draw[color=red, thick] (cc3) -- (cc4);
        \draw[color=red, thick] (cc2) -- (dd2);
        \draw[color=red, thick] (cc2) -- (dd22);
        \draw[color=red, thick] (cc2) -- (dd23);
        \draw[color=red, thick] (cc4) -- (dd4);
        \node[fill=none, scale=2.5] at (-1.2,1) {$T=$};
        \node[fill=none, scale=2.5] at (8.5,1) {$\mathcal{I}_T=$};
    \end{tikzpicture}}
    \label{fig:internal-subgraph}
\end{example}
\vskip 0in 
In \cite{GOT-CSF}, we showed that the internal subgraph of a tree is always connected. 
\vskip 0in
\subsection{Leaf components in the internal subgraph} From $\lead(\mathbf{X}_T)$ for a tree $T$, we can recover the orders of the leaf components of $T$. When $T$ has diameter 4, only one of these leaf components is in $\mathcal{I}_T$; and if $T$ has diameter 5, then $\mathcal{I}_T$ contains only two leaf components of $T$. In the case where $T$ has diameter at most 5, we can find from $\mathbf{X}_T$ the orders of the leaf components in $T$ that are also in $\mathcal{I}_T$. Leaf components of order 1 in $T$ must be in $\mathcal{I}_T$. Hence, we concentrate on leaf components of order greater than 1.

We now define a number that will help us determine the orders of the leaf components of $T$ that are contained in $\mathcal{I}_T$. If $\csft=\sum_\lambda c_\lambda \stfrak_\lambda$ is the CSF of a tree $T$ of order $n$ and $p$ is a part in $\lead(\csft)$, then define the quantity:
\begin{equation}\label{eq:N(p)}
    N(p) := \sum m_p(E_\mu) \cdot c_\mu ~,
\end{equation}   
where the sum runs over all $\mu\vdash n$ of length $\ell(\lead(\csft))-1$ such that $c_\mu \neq 0$ in $\csft$ and such that $\mu$ does not contain 1 as a part. Recall $m_p({E_\mu})$ is the multiplicity of $p$ in $E_\mu$. From Equation (\refeq{eq:N(p)}), we observe that $N(p)$ can be computed from $\mathbf{X}_T$. 

\begin{remark}\label{remark:n(p)-interpretation}
    If $\lead(\csft)$ contains no 1s, then Proposition \ref{prop:k-edge-adjacencies}(a) implies that $\mathcal{E}_T=\bigsqcup_\mu E_\mu$, where the union is taken over all $\mu \vdash n$ of length $\ell(\lead(\csft))-1$ such that $c_\mu\neq 0$ in $\csft$ and $\mu$ contains no 1s, contains the adjacencies between leaf components in $T$. Further, the multiplicity of $E_\mu$ in $\mathcal{E}_T$ is $c_\mu$.
    Hence, $N(p)$ as defined above is exactly the number of times that a leaf component of order $p$ occurs as a leaf component endpoint in $T$. 
\end{remark}

\begin{theorem}[\cite{GOT-CSF}, Theorem 5.19]\label{thm:leaf-component-order}
    Let $\csft$ be the CSF of a tree $T$ of order $n$ with no deep vertices, and let $\lead(\csft) = (n^{m_n}, \ldots, 2^{m_2})$ be the leading partition. If $p$ is any part of $\lead(\csft)$, then a leaf component of order $p$ in $T$ is contained in $\mathcal{I}_T$ if and only if $N(p) > m_p(\lead(\csft))=m_p$.
\end{theorem}

\subsubsection{Reconstruction of trees of diameter 4.} \label{subsec:diameter 4}
The following theorem is a consequence of Theorem \ref{Thm:lead}, which says that the leading partition gives us the orders of all the leaf components of $T$, and the fact that $\mathcal{I}_T$ has only one leaf component, $\mathcal{L}$, when $T$ has diameter four \cite[Corollary 5.7]{GOT-CSF}, we can recover the order of $\mathcal{L}$ from Theorem \ref{thm:leaf-component-order} or from $\lead(\csft)$ if $m_1(\lead(\csft)) =1$.

\begin{theorem}\label{thm:diam-4-reconstruction}
    Trees of diameter four can be reconstructed from their CSF. In particular, these trees can be reconstructed from the leading partition, $\lead(\csft)$, and the nonzero coefficients $c_\mu$, indexed by partitions $\mu$ such that $\ell (\mu) = \ell(\lead(\csft))-1$ and $\mu$ has no parts of size 1. 
\end{theorem}
 In fact, we can give an algorithm to reconstruct trees of diameter 4. Given $\csft$ in the star basis, the leading partition $\lead(\csft)$ gives the orders of the leaf-components of $T$. Since $T$ has diameter 4, $\mathcal{I}_T$ is a single leaf component and $m_1(\lead(\csft)) \leq 1$.
\begin{itemize}
    \item If $\lead(\csft)=(\lambda_1, \ldots, \lambda_\ell)$ and $m_1 =1$, then $\mathcal{I}_T$ is a single vertex, $v$, then $T$ is the tree obtained by adding an edge from $v$ the center of a star of order $\lambda_i$ for all $1\leq i < \ell$.
    \item If $\lead(\csft)=(\lambda_1, \ldots, \lambda_\ell)$ and $m_1=0$, then:
    \begin{itemize}
        \item Use Theorem \ref{thm:leaf-component-order} to determine the order, $\lambda_j$, of the leaf component in $\mathcal{I}_T$, where $1\leq j \leq \ell$. 
        \item Then add an edge from the central vertex of a star of order $\lambda_i$, for $i\neq j$, to the central vertex of $\mathcal{I}_T$.
    \end{itemize}
\end{itemize}

\begin{example}\label{ex:diam-4-reconstruction}
Consider the following CSF on a tree with 17 vertices.
{\small \begin{eqnarray*}
    \csft &=&  \textcolor{blue}{\stfrak_{(5, 4, 3, 3, 2)}}- \stfrak_{(5, 5, 3, 3, 1)} - 2\stfrak_{(6, 5, 3, 2, 1)} + \textcolor{red}{\stfrak_{(6, 5, 3, 3)}} + 2\stfrak_{(7, 5, 3, 1, 1)} + \textcolor{red}{2\stfrak_{(7, 5, 3, 2)}}  - \stfrak_{(8, 3, 3, 2, 1)} +
    \\
    && \stfrak_{(8, 5, 2, 1, 1)} - 4\stfrak_{(8, 5, 3, 1)} + \stfrak_{(9, 3, 3, 1, 1)} + \textcolor{red}{\stfrak_{(9, 3, 3, 2)}} - \stfrak_{(9, 5, 1, 1, 1)} - 2\stfrak_{(9, 5, 2, 1)} + 2\stfrak_{(9, 5, 3)}+
    \\
    && 2\stfrak_{(10, 3, 2, 1, 1)} - 2\stfrak_{(10, 3, 3, 1)} + 3\stfrak_{(10, 5, 1, 1)} + \stfrak_{(10, 5, 2)} - 2\stfrak_{(11, 3, 1, 1, 1)} - 4\stfrak_{(11, 3, 2, 1)} + \stfrak_{(11, 3, 3)}
    \\
    && - 3\stfrak_{(11, 5, 1)} - \stfrak_{(12, 2, 1, 1, 1)} + 6\stfrak_{(12, 3, 1, 1)}+ 2\stfrak_{(12, 3, 2)} + \stfrak_{(12, 5)}+ \stfrak_{(13, 1, 1, 1, 1)}+ 3\stfrak_{(13, 2, 1, 1)}
    \\
     &&   - 6\stfrak_{(13, 3, 1)}- 4\stfrak_{(14, 1, 1, 1)}
     - 3\stfrak_{(14, 2, 1)} + 2\stfrak_{(14, 3)}+ 6\stfrak_{(15, 1, 1)}  + \stfrak_{(15, 2)} - 4\stfrak_{(16, 1)} +\stfrak_{(17)} 
\end{eqnarray*} }

As shown in \cite{martin2008distinguishing}, we know that $T$ has diameter four from $\mathbf{X}_T$. We have $\lead(\csft) = (5,4,3,3,2)$, colored blue. The terms indexed by partitions of length $\ell(\lead(\csft))-1$ without 1s are colored red. These terms induce the adjacency multisets $\lmulti 5,4\rmulti, \lmulti 4, 3\rmulti$ (twice) and $\lmulti 4, 2\rmulti$. We have $N(4)=4>1=m_4$, which by Theorem \ref{thm:leaf-component-order} implies that the leaf component in $T$ contained in $\mathcal{I}_T$ has order $4$. Then, we draw all the remaining leaf components and connect them to the leaf component of order 4 in $\mathcal{I}_T$. 

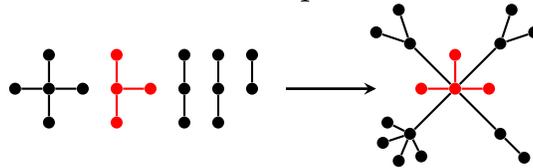
\centerline{
    \begin{tikzpicture}[auto=center,every node/.style={circle, fill=black,scale=0.4}, thick, scale=0.6]
        
        \node (A1) at (7,0) {};
        \node (A2) at (7,0.75) {};
        \node (A3) at (7,-0.75) {};
        \node (A4) at (7.75,0) {};
        \node (A5) at (6.25,0) {};
        
        \draw(A1) -- (A2);
        \draw(A1) -- (A3);
        \draw(A1) -- (A4);
        \draw(A1) -- (A5);
        
        \node[fill=red] (B1) at (8.5, 0) {};
        \node[fill=red] (B2) at (8.5, 0.75) {};
        \node[fill=red] (B3) at (8.5, -0.75) {};
        \node[fill=red] (B4) at (9.25, 0) {};
        
        \draw[red](B1) -- (B2);
        \draw[red](B1) -- (B3);
        \draw[red](B1) -- (B4);
        
        \node (C1) at (10,0) {};
        \node (C2) at (10,0.75) {};
        \node (C3) at (10,-0.75) {};
        
        \draw(C1) -- (C2);
        \draw(C1) -- (C3);
        
        \node (D1) at (10.75,0) {};
        \node (D2) at (10.75,0.75) {};
        \node (D3) at (10.75,-0.75) {};
        
        \draw(D1) -- (D2);
        \draw(D1) -- (D3);
        
        \node (E1) at (11.5,0) {};
        \node (E2) at (11.5,0.75) {};
        
        \draw(E1) -- (E2);
        
        \draw[-stealth, line width=1pt](12.25,0) -- (14.25,0);
        
        \node[fill=red] (A1) at (16,0) {};
        \node[fill=red] (A2) at (16.75,0) {};
        \node[fill=red] (A3) at (15.25,0) {};
        \node[fill=red] (A4) at (16,0.75) {};
        
        \draw[red](A1) -- (A2);
        \draw[red](A1) -- (A3);
        \draw[red](A1) -- (A4);
        
        \node (B1) at (17,1) {};
        \node (B2) at (17.25,1.75) {};
        \node (B3) at (17.75,1.25) {};
        
        \draw(B1) -- (B2);
        \draw(B1) -- (B3);

        \node (C1) at (15,1) {};
        \node (C2) at (14.75, 1.75) {};
        \node (C3) at (14.25, 1.25){};

        \draw(C1) -- (C2);
        \draw(C1) -- (C3);

        \node (D1) at (17,-1) {};
        \node (D2) at (17.53, -1.53){};
        
        \draw(D1) -- (D2);
        
        \node (E1) at (15,-1){};
        \node (E2) at (14.75,-1.6) {};
        \node (E3) at (14.4, -1.2) {};
        \node (E4) at (15.25, -1.5) {};
        \node (E5) at (14.5, -0.75) {};
        
        \draw(E1) -- (E2);
        \draw(E1) -- (E3);
        \draw(E1) -- (E4);
        \draw(E1) -- (E5);
        \draw(A1) -- (B1);
        \draw(A1) -- (C1);
        \draw(A1) -- (D1);
        \draw(A1) -- (E1);    
    \end{tikzpicture}}
\end{example}

\subsubsection{Trees of diameter 5}

\begin{lemma}\label{lemma: diameter of T1 and T2}
    Let $T$ be a tree with diameter five and internal subgraph $\mathcal{I}_T$. Let $e=u_1u_2$ be the edge between the two internal vertices $u_1$  and  $u_2$ contained in $\mathcal{I}_T$. Let $T\setminus e = T_1 \sqcup T_2$ where $T_1$ and $T_2$ are the trees containing $u_1$ and $u_2$, respectively. Then both $T_1$ and $T_2$ have diameter at most 4.
\end{lemma}

\begin{proposition}\cite[Proposition 5.26]{GOT-CSF}\label{prop:xt1 and xt2}
    Let $\csft$ be the CSF of a tree $T$ with diameter five and $\lead = \lead(\csft)=(n^{m_n}, \ldots, 1^{m_1})$. Let $e$ be the internal edge in $\mathcal{I}_T$, the internal subgraph of $T$, and let $T\setminus e = T_1 \sqcup T_2$. Then, $\mathbf{X}_{T_1}\mathbf{X}_{T_2}$ can be recovered from $\csft$. Further, $\#V(T_1)$ and $\#V(T_2)$ can be recovered from $\csft$.
\end{proposition}

In \cite{GOT-CSF}, we prove the result below. However, we include here a shorter inductive proof that did not appear in \cite{GOT-CSF}.

\begin{theorem}
    Let $\csft$ be the CSF of a tree $T$ with diameter 5 and $\lead = \lead(\csft)=(n^{m_n}, \ldots, 1^{m_1})$. Then, $T$ can be reconstructed from $\csft$.
\end{theorem}

\begin{proof}
    Since $T$ has diameter 5, $m_1\leq 2$ and the internal subgraph $\calit$ consists of two leaf components $\mathcal{L}_1$ and $\mathcal{L}_2$ of orders $p_1$ and $p_2$, respectively, whose respective centers $u_1$ and $u_2$ are joined by an internal edge $e$. Let $T \setminus e = T_1 \sqcup T_2$, where $T_1$ contains $u_1$ and $T_2$ contains $u_2$. Then, by Proposition \ref{prop:xt1 and xt2}, we can recover the product $\mathbf{X}_{T_1}\mathbf{X}_{T_2}$ as well as $\#V(T_1)$ and $\#V(T_2)$. Let $N_1 = \#V(T_1)$ and $N_2 = \#V(T_2)$. If $\lead$ has only two parts greater than one, then $T$ is reconstructible by Corollary \ref{cor:bi-stars-leading}. Thus, we may assume that $\lead$ has at least three parts greater than one.

    Without loss of generality, we may assume $\lead(\mathbf{X}_{T_1}) \leq \lead(\mathbf{X}_{T_2})$ where $\leq$ is lexicographic order. Therefore, in the product $\mathbf{X}_{T_1}\mathbf{X}_{T_2}$, the smallest partition $\alpha$ in lexicographic order such that $c_{(N_2,\alpha)} \neq 0$ must be exactly $\alpha = \lead(\mathbf{X}_{T_1})$. Thus, we can recover the orders of the leaf components in $T_1$, and thus those in $T_2$ by taking the multiset difference with $\lead$. Hence, we can recover $\lead(\mathbf{X}_{T_1})$ and $\lead(\mathbf{X}_{T_2})$ from $\mathbf{X}_{T_1}\mathbf{X}_{T_2}$. Observe that we can tell whether $p_1 = p_2$ from either $\lead$ if $m_1(\lead)=2$ or from Theorem \ref{thm:leaf-component-order} if $m_1(\lead)=0$. In this case, since $p_1 = p_2$, we can reconstruct $T$ by adding an edge between the centers of $\mathcal{L}_1$ and $\mathcal{L}_2$, then adding edges from the centers of the other leaf components given by $\lead(\mathbf{X}_{T_1})-\{p_1\}$ to $\mathcal{L}_1$, and finally adding edges from the center of $\mathcal{L}_2$ to the centers of the components given by the mulitset difference $(\lead - \lead(\mathbf{X}_{T_1}))-\lmulti p_2\rmulti=\lead(\mathbf{X}_{T_2})-\lmulti p_2\rmulti$.
    
    Hence, assume $p_1 \neq p_2$. If $m_1(\lead) =1$, we consider the following cases. If $\ell(\lead(\mathbf{X}_{T_1}))=1$, then $T_1$ is a star on $k=|\lead(\mathbf{X}_{T_1})|$ vertices. Then, $p_1 =1$ and the only two leaf components adjacent to $\mathcal{L}_1$ have orders $k-1$ and $p_2$. In this case, we can find $p_2$ using adjacency multisets from Proposition \ref{prop:k-edge-adjacencies}(b). If $\ell(\lead(\mathbf{X}_{T_1})) \neq 1$, then either $p_1=1$ if $m_1(\lead(\mathbf{X}_{T_1}))=1$ and we can find $p_2$ again using adjacency multisets. Otherwise, we must have $p_2 = 1$. Since we recovered the orders of the leaf components in $T_2$, using the adjacencies from Proposition \ref{prop:k-edge-adjacencies}(b) we can find the value of $p_1$. Once we know the values of $p_1$ and $p_2$ the reconstruction of $T$ is the same as when $p_1 = p_2$.

    If $m_1(\lead) =0$, then by Theorem \ref{thm:leaf-component-order} we can find the multiset of values of the orders of $\mathcal{L}_1$ and $\mathcal{L}_2$, $\lmulti r, s \rmulti$. We need to determine if $p_2 = r$ or $s$. If only $r$ (or $s$) occurs in $\lead(\mathbf{X}_{T_1})$ (similarly for $\lead(\mathbf{X}_{T_2})$), then $p_1=r$ (or $s$) and $p_2=s$ (or $r$). Then, assume $\lead(\mathbf{X}_{T_1}) = \sort (r,s,a_1,\ldots,a_m)$ and $\lead(\mathbf{X}_{T_2}) = \sort (r,s, b_1,\ldots,b_k)$. Then, $T$ must be equal to one of the following trees (each leaf component is represented by its order):

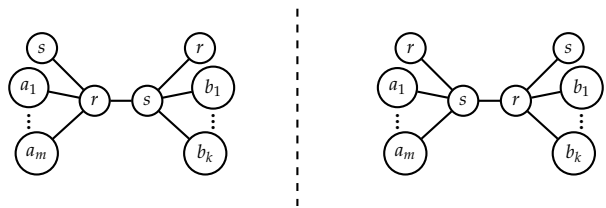
\centerline{
    \begin{tikzpicture}[auto=center,every node/.style={circle, fill=none, draw=black, scale=0.6}, thick, scale=0.7]
        
        \node (a1) at (0,0) {$r$};
        \node (a2) at (1,0) {$s$};
        \node (b1) at (-1, 1) {$s$};
        \node (a0) at (-1.25,0.25) {$a_1$};
        \node[draw=none] at (-1.25, -.3) {\Large{$\vdots$}};
        \node (c1) at (-1.1, -1) {$a_m$};
        \node (d1) at (2, 1) {$r$};
        \node (b2) at (2.25, 0.25) {$b_1$};
        \node[draw=none] at (2.25,-.3) {\Large{$\vdots$}};
        \node (c2) at (2.1, -1) {$b_k$};
        
        \draw (a0) -- (a1);
        \draw (a1) -- (a2);
        \draw (a1) -- (b1);
        \draw (a1) -- (c1);        
        \draw (a2) -- (b2);
        \draw (a2) -- (c2);
        \draw (a2)--(d1);

        \draw[dashed] (3.85, 1.75) -- (3.85, -2);

        \node (a1) at (7,0) {$s$};
        \node (a2) at (8,0) {$r$};
        \node (b1) at (6, 1) {$r$};
        \node (e1) at (9, 1) {$s$};
        \node (a0) at (5.75,0.25) {$a_1$};
        \node[draw=none] at (5.75, -.3) {\Large{$\vdots$}};
        \node (c1) at (5.9, -1) {$a_m$};
        \node (b2) at (9.25, 0.25) {$b_1$};
        \node[draw=none] at (9.25,-.3) {\Large{$\vdots$}};
        \node (c2) at (9.1, -1) {$b_k$};
        
        \draw (a0) -- (a1);
        \draw (a1) -- (a2);
        \draw (a1) -- (b1);
        \draw (a1) -- (c1);        
        \draw (a2) -- (b2);
        \draw (a2) -- (c2);
        \draw (e1)--(a2);
    \end{tikzpicture}}

    But it is clear that if $\lead(\mathbf{X}_{T_1})\neq \lead(\mathbf{X}_{T_2})$, these trees have different edge adjencencies, otherwise they are isomorphic. 
\end{proof}



\printbibliography
\end{document}